\documentclass[12pt]{article}

\usepackage{latexsym,amsmath,amscd,amssymb,graphics,float}
\usepackage{enumerate}

\usepackage{graphicx,lscape}

\usepackage[colorlinks]{hyperref}
\usepackage{url}

\begin{document}

\newenvironment{proof}[1][Proof]{\textbf{#1.} }{\ \rule{0.5em}{0.5em}}

\newtheorem{theorem}{Theorem}[section]
\newtheorem{definition}[theorem]{Definition}
\newtheorem{lemma}[theorem]{Lemma}
\newtheorem{remark}[theorem]{Remark}
\newtheorem{proposition}[theorem]{Proposition}
\newtheorem{corollary}[theorem]{Corollary}
\newtheorem{example}[theorem]{Example}

\numberwithin{equation}{section}
\newcommand{\ep}{\varepsilon}
\newcommand{\R}{{\mathbb  R}}
\newcommand\C{{\mathbb  C}}
\newcommand\Q{{\mathbb Q}}
\newcommand\Z{{\mathbb Z}}
\newcommand{\N}{{\mathbb N}}

\newcommand{\bfi}{\bfseries\itshape}

\newsavebox{\savepar}
\newenvironment{boxit}{\begin{lrbox}{\savepar}
\begin{minipage}[b]{15.5cm}}{\end{minipage}\end{lrbox}
\fbox{\usebox{\savepar}}}

\title{{\bf On invariant properties of natural differential operators associated to geometric structures on $\mathbb{R}^n$}}
\author{R\u{a}zvan M. Tudoran}

\date{}
\maketitle \makeatother

\begin{abstract}
We provide a general framework to study invariant properties of various gradient-like and Laplace-like differential operators naturally associated to geometric structures on $\mathbb{R}^n$, which encompass Euclidean, Minkowski, pseudo-Euclidean and symplectic structures.
\end{abstract}

\medskip

\textbf{MSC 2020}: 47A07; 47D03; 47F05.

\textbf{Keywords}: geometric structures; gradient-like vector fields; invariant functions; Laplace-like operators.

\section{Introduction}
\label{section:one}

The aim of this article is to provide a general framework to study invariant properties of various gradient-like and Laplace-like differential operators naturally associated to geometric structures on $\mathbb{R}^n$, which encompass Euclidean, Minkowski, pseudo-Euclidean and symplectic structures. In this regard we shall call \textit{geometric structure} any non-degenerate bilinear form $b$ on $\mathbb{R}^n$, and we will denote by $\mathbb{R}^{n}_{b}$ the pair $(\mathbb{R}^n, b)$.

The structure of the article is the following. In the second section we recall from \cite{TDR} the definition of geometric structures and also a natural representation of any geometric structure in terms of the canonical Euclidean inner product on $\mathbb{R}^n$. In the third section we define the notion of left/right-adjoint operator induced by a geometric structure and also we provide some natural properties of these operators. The fourth section is devoted to the study of the kernel and the image of left/right-adjoint operators induced by a general geometric structure. In the fifth section we associate to any fixed geometric structure $b$ a natural group, denoted by $G_b$, consisting of the set of linear operators preserving the geometric structure $b$, and we provide two equivalent definitions of this group in terms of left/right-adjoint operators. In the sixth section we introduce two gradient-like operators naturally associated to any given geometric structure $b$, in order to study the compatibility between smooth $G_b$-invariant scalar functions and smooth $G_b$-equivariant vector fields on $\mathbb{R}^{n}_{b}$. The main result of this section provides explicitly in terms of left/right-gradient vector fields, two sets of generators for the module of smooth $G_b$-equivaraint vector fields on $\mathbb{R}^{n}_{b}$. The aim of the last section is to introduce and study $G_b$-invariant properties of a Laplace-like operator naturally associated to any given geometric structure on $\mathbb{R}^n$. This operator generalizes the classical Laplace operator in the case when the geometric structure is Euclidean, and also the d'Alembert operator in the case when the geometric structure is Minkowski.

\section{Natural geometric structures on $\mathbb{R}^n$}

The goal of this short section is to recall from \cite{TDR} the definition of \textit{geometric structures}, motivated by the observation that the classical geometries on $\mathbb{R}^n$ (i.e. Euclidean, Minkowski, pseudo-Euclidean, symplectic) are each generated by some nondegenerate real bilinear form with specific properties. 

In order to have a general setting that encompasses all of the above mentioned geometries, we introduced in \cite{TDR} the notion of \textit{geometric structure}. Specifically, by \textit{geometric structure} on $\mathbb{R}^{n}$ as defined in \cite{TDR}, we mean a fixed nondegenerate real bilinear form on $\mathbb{R}^n$. The pair $(\mathbb{R}^{n},b)$, where $b$ is a geometric structure on $\mathbb{R}^n$, is denoted by $\mathbb{R}^{n}_{b}$.

This definition of geometric structure includes naturally the Euclidean, Minkowski, pseudo-Euclidean, and symplectic geometries: if $b$ is symmetric and positive definite, then $\mathbb{R}^{n}_{b}$ is an Euclidean vector space, if $n$ is even and $b$ is skew--symmetric, then $\mathbb{R}^{n}_{b}$ is a symplectic vector space, if $b$ is symmetric with the signature $(n-k,k)$, $k\neq 0$, then $\mathbb{R}^{n}_{b}$ is a pseudo-Euclidean vector space (if $k=1$ then $\mathbb{R}^{n}_{b}$ is a Minkowski vector space).

Let us recall now a result from \cite{TDR} which is essential for the rest of the article. More precisely, for any arbitrary fixed geometric structure $b$ on $\mathbb{R}^{n}$, there exists a \textit{unique} invertible linear map $B\in\operatorname{Aut}(\mathbb{R}_{b}^{n})$ such that 
\begin{equation}\label{relb}
\langle \mathbf{x},\mathbf{y} \rangle = b(\mathbf{x}, B\mathbf{y}), ~~\forall \mathbf{x},\mathbf{y}\in\mathbb{R}^{n}_{b},
\end{equation}
where $\langle \cdot,\cdot \rangle$ stands for the canonical Euclidean inner product on $\mathbb{R}^{n}$. The pair $(b,B)$ is called the \textit{geometric pair} on $\mathbb{R}^{n}$ generated by the geometric structure $b$.

\section{The left/right - adjoint operators}

Given a geometric structure $b$ on $\mathbb{R}^n$, for any operator $A\in\operatorname{End}(\mathbb{R}_{b}^{n})$ we define two natural operators compatible with the geometric structure. These operators generalize the classical adjoint operator with respect to an Euclidean or pseudo-Euclidean structure.
  
\begin{definition}\label{starLR}
Let $b$ be a geometric structure on $\mathbb{R}^n$ and $A\in\operatorname{End}(\mathbb{R}_{b}^{n})$. Then there exist two linear maps, denoted by $A^{\star_{L}},A^{\star_{R}}\in\operatorname{End}(\mathbb{R}_{b}^{n})$ and called the \textit{left-adjoint}, respectively the \textit{right-adjoint} of $A$ with respect to the geometric structure $b$, uniquely defined by the following relations:
\begin{equation*}
b(A^{\star_{L}}\mathbf{x}, \mathbf{y})= b(\mathbf{x}, A\mathbf{y}), ~~\forall \mathbf{x},\mathbf{y}\in\mathbb{R}^{n}_{b},
\end{equation*}
\begin{equation*}
b(A\mathbf{x}, \mathbf{y})= b(\mathbf{x}, A^{\star_{R}}\mathbf{y}), ~~\forall \mathbf{x},\mathbf{y}\in\mathbb{R}^{n}_{b}.
\end{equation*}
\end{definition}
In terms of geometric pairs \eqref{relb} we obtain the following characterization of the left/right-adjoint operators.
\begin{proposition}\label{pimp}
Let $b$ be a geometric structure on $\mathbb{R}^n$ and let $(b,B)$ be the associated geometric pair. Then for each  $A\in\operatorname{End}(\mathbb{R}_{b}^{n})$ the following assertions hold:
\begin{itemize}
\item[(i)] $A^{\star_{L}}=B^{\top}A^{\top}(B^{\top})^{-1}$,
\item[(ii)]$A^{\star_{R}}=B A^{\top} B^{-1}$,
\item[(iii)] $\operatorname{Id}^{\star_L}=\operatorname{Id}^{\star_R}=\operatorname{Id}$, $B^{\star_{L}}=B^{\top}$, $B^{\star_{R}}=B B^{\top} B^{-1}$,
\end{itemize}
where $\top$ stands for the adjoint with respect to $\langle\cdot,\cdot\rangle$, the canonical Euclidean inner product on $\mathbb{R}^{n}$.
\end{proposition}
\begin{proof} From the definition of a geometric pair \eqref{relb} we obtain that 
\begin{equation}\label{use}
b(\mathbf{x}, \mathbf{y})=\langle \mathbf{x},B^{-1}\mathbf{y} \rangle, ~~\forall \mathbf{x},\mathbf{y}\in\mathbb{R}^{n}_{b}.
\end{equation}
\begin{itemize}
\item[(i)] Using the equality \eqref{use}, the relation $b(A^{\star_{L}}\mathbf{x}, \mathbf{y})= b(\mathbf{x}, A\mathbf{y}), \forall \mathbf{x},\mathbf{y}\in\mathbb{R}^{n}_{b}$, becomes
\begin{align*}
\langle A^{\star_{L}}\mathbf{x},B^{-1}\mathbf{y} \rangle =\langle \mathbf{x},B^{-1}A\mathbf{y} \rangle, \forall \mathbf{x},\mathbf{y}\in\mathbb{R}^{n}_{b},
\end{align*}
which is equivalent to
\begin{align*}
\langle (B^{-1})^{\top}A^{\star_{L}}\mathbf{x},\mathbf{y} \rangle =\langle A^{\top}(B^{-1})^{\top}\mathbf{x},\mathbf{y} \rangle, \forall \mathbf{x},\mathbf{y}\in\mathbb{R}^{n}_{b}.
\end{align*}
Consequently we get that $(B^{-1})^{\top}A^{\star_{L}}=A^{\top}(B^{-1})^{\top}$ and hence $A^{\star_{L}}=B^{\top}A^{\top}(B^{\top})^{-1}$.
\item[(ii)] The proof is similar to that of the previous item, this time using the equality \eqref{use} and the relation
\begin{equation*}
b(A\mathbf{x}, \mathbf{y})= b(\mathbf{x}, A^{\star_{R}}\mathbf{y}), ~~\forall \mathbf{x},\mathbf{y}\in\mathbb{R}^{n}_{b}.
\end{equation*}
\item[(iii)] The proof follows directly from $(i)$ and $(ii)$.
\end{itemize}
\end{proof}

Let us now emphasize some natural properties of the left/right - adjoint operators following directly from Proposition \ref{pimp}.
\begin{remark}
Let $b$ be a geometric structure on $\mathbb{R}^n$ and let $(b,B)$ be the associated geometric pair. Then the left/right - adjoint operators verify the following identities:
\begin{itemize}
\item[(i)] $(\lambda_1 A_1+\lambda_2 A_2)^{\star_{L}}=\lambda_1 A_1 ^{\star_{L}}+\lambda_2 A_2 ^{\star_{L}}$, $(\lambda_1 A_1+\lambda_2 A_2)^{\star_{R}}=\lambda_1 A_1 ^{\star_{R}}+\lambda_2 A_2 ^{\star_{R}}$, $\forall \lambda_1, \lambda_2 \in\mathbb{R}$, $\forall A_1, A_2 \in\operatorname{End}(\mathbb{R}_{b}^{n})$,
\item[(ii)] $(A^{\star_{L}})^{\star_{R}}=(A^{\star_{R}})^{\star_{L}}=A$, $\forall A\in\operatorname{End}(\mathbb{R}_{b}^{n})$,
\item[(iii)] $(A^{\star_{L}})^{\star_{L}}=A$, $\forall A\in\operatorname{End}(\mathbb{R}_{b}^{n})$ if and only if $B^{\top}=\varepsilon B$ with $\varepsilon^n=1$, i.e. when $b$ is symmetric ($\varepsilon=1$) or skew-symmetric ($\varepsilon=-1$ and $n$ is even),
\item[(iv)] $(A^{\star_{R}})^{\star_{R}}=A$, $\forall A\in\operatorname{End}(\mathbb{R}_{b}^{n})$ if and only if $B^{\top}=\varepsilon B$ with $\varepsilon^n=1$, i.e. when $b$ is symmetric ($\varepsilon=1$) or skew-symmetric ($\varepsilon=-1$ and $n$ is even),
\item[(v)] $A^{\star_{L}}=A^{\star_{R}}$, $\forall A\in\operatorname{End}(\mathbb{R}_{b}^{n})$ if and only if $B^{\top}=\varepsilon B$ with $\varepsilon^n=1$, i.e. when $b$ is symmetric ($\varepsilon=1$) or skew-symmetric ($\varepsilon=-1$ and $n$ is even),
\item[(vi)] $(A_1 A_2)^{\star_{L}}=A_2 ^{\star_{L}} A_1 ^{\star_{L}}$, $(A_1 A_2)^{\star_{R}}=A_2 ^{\star_{R}} A_1 ^{\star_{R}}$, $\forall A_1, A_2 \in\operatorname{End}(\mathbb{R}_{b}^{n})$,
\item[(vii)] $(A^{-1})^{\star_L}=(A^{\star_L})^{-1}$, $(A^{-1})^{\star_R}=(A^{\star_R})^{-1}$, $\forall A\in\operatorname{Aut}(\mathbb{R}_{b}^{n})$,
\item[(viii)] $[A_1,A_2]^{\star_{L}}=[A_2 ^{\star_{L}},A_1 ^{\star_{L}}]$, $[A_1,A_2]^{\star_{R}}=[A_2 ^{\star_{R}},A_1 ^{\star_{R}}]$, $\forall A_1, A_2 \in\operatorname{End}(\mathbb{R}_{b}^{n})$, where $[\cdot,\cdot]$ stands for the standard commutator of operators, i.e. $[A_1,A_2]=A_1 A_2 - A_2 A_1$, $\forall A_1, A_2 \in\operatorname{End}(\mathbb{R}_{b}^{n})$.
\end{itemize}
\end{remark}

\section{The kernel and image of a left/right - adjoint operator}

Two of the most important subspaces associated to a general operator are the kernel and the image. The aim of this section is to analyze these subspaces associated to the left/right - adjoint operators. In order to do that let us introduce first some terminology. 

More precisely, given a geometric structure $b$ and a vector subspace $V\subseteq \mathbb{R}^{n}_{b}$, there exist two (isomorphic) vector subspaces, $V^{\perp_L}=\{w\in\mathbb{R}^{n}_{b}: ~ b(w,v)=0, \forall v\in V \}$ and  $V^{\perp_R}=\{w\in\mathbb{R}^{n}_{b}: ~ b(v,w)=0, \forall v\in V \}$, both of dimension $n-\operatorname{dim}_{\mathbb{R}} V$ (note that if $b$ is symmetric or skew-symmetric, then $V^{\perp_L}=V^{\perp_R}$).

Counting dimensions, it follows that for any vector subspace $V\subseteq \mathbb{R}^{n}_{b}$,   $$\operatorname{dim}_{\mathbb{R}}(V^{\perp_L})^{\perp_R}=\operatorname{dim}_{\mathbb{R}}(V^{\perp_R})^{\perp_L}=\operatorname{dim}_{\mathbb{R}}V,$$thus, taking into account the natural inclusions $V\subseteq (V^{\perp_L})^{\perp_R}$, $V\subseteq (V^{\perp_R})^{\perp_L}$, we get
\begin{equation}\label{orLR}
(V^{\perp_L})^{\perp_R}=(V^{\perp_R})^{\perp_L}=V.
\end{equation}
Let us now return to the characterization of the kernel and image of the left/right - adjoint operators. The following result provides some natural relations concerning the image and the kernel of the operators $A, A^{\star_L}$ and $A^{\star_R}$, for any given $A\in\operatorname{End}(\mathbb{R}_{b}^{n})$.
\begin{theorem}
Let $b$ be a geometric structure on $\mathbb{R}^n$ and $A\in\operatorname{End}(\mathbb{R}_{b}^{n})$. Then the following relations hold true:
\begin{itemize}
\item[(i)] $\operatorname{Ker} A^{\star_L}=(\operatorname{Im} A)^{\perp_L}$, $\operatorname{Ker} A^{\star_R}=(\operatorname{Im} A)^{\perp_R}$, $(\operatorname{Ker} A^{\star_L})^{\perp_R}=(\operatorname{Ker} A^{\star_R})^{\perp_L}$,
\item[(ii)] $\operatorname{Im} A^{\star_L}=(\operatorname{Ker} A)^{\perp_L}$, $\operatorname{Im} A^{\star_R}=(\operatorname{Ker} A)^{\perp_R}$, $(\operatorname{Im} A^{\star_L})^{\perp_R}=(\operatorname{Im} A^{\star_R})^{\perp_L}$,
\item[(iii)] $\operatorname{Ker} A=(\operatorname{Im} A^{\star_L})^{\perp_R}=(\operatorname{Im} A^{\star_R})^{\perp_L}$, $\operatorname{Im} A=(\operatorname{Ker} A^{\star_L})^{\perp_R}=(\operatorname{Ker} A^{\star_R})^{\perp_L}$.
\end{itemize}
\end{theorem}
\begin{proof}
\begin{itemize}
\item[(i)] Let us start by proving that $(\operatorname{Im} A)^{\perp_L}=\operatorname{Ker} A^{\star_{L}}$. This equality is obtained directly from the following equivalences
\begin{align*}
w\in(\operatorname{Im} A)^{\perp_L} &\Leftrightarrow [ b(w,Av)=0, \forall v\in\mathbb{R}^{n}_{b}] \Leftrightarrow [b(A^{\star_L}w,v)=0, \forall v\in\mathbb{R}^{n}_{b}] \\
&\Leftrightarrow A^{\star_L}w=0\Leftrightarrow w\in\operatorname{Ker} A^{\star_{L}}.
\end{align*}
Similarly we get that $(\operatorname{Im} A)^{\perp_R}=\operatorname{Ker} A^{\star_R}$.

Using the identities $\operatorname{Ker} A^{\star_L}=(\operatorname{Im} A)^{\perp_L}$, $\operatorname{Ker} A^{\star_R}=(\operatorname{Im} A)^{\perp_R}$ and the relations \eqref{orLR} we obtain
\begin{align*}
(\operatorname{Ker} A^{\star_L})^{\perp_R}=((\operatorname{Im} A)^{\perp_L})^{\perp_R}=((\operatorname{Im} A)^{\perp_R})^{\perp_L}=(\operatorname{Ker} A^{\star_R})^{\perp_L}.
\end{align*}
\item[(ii)] From the second equality of item $(i)$ together with relations \eqref{orLR} it follows that
\begin{align*}
(\operatorname{Ker} A^{\star_R})^{\perp_L}=((\operatorname{Im} A)^{\perp_R})^{\perp_L}\Leftrightarrow (\operatorname{Ker} A^{\star_R})^{\perp_L}=\operatorname{Im} A.
\end{align*}
Replacing $A$ by $A^{\star_L}$ in the equality $\operatorname{Im} A=(\operatorname{Ker} A^{\star_R})^{\perp_L}$ and taking into account that $(A^{\star_L})^{\star_R}=A$ we get
\begin{align*}
\operatorname{Im} A^{\star_L}=(\operatorname{Ker} (A^{\star_L})^{\star_R})^{\perp_L}\Leftrightarrow \operatorname{Im} A^{\star_L}=(\operatorname{Ker} A)^{\perp_L}.
\end{align*}
Similarly we obtain $\operatorname{Im} A^{\star_R}=(\operatorname{Ker} A)^{\perp_R}$.

Using the identities $\operatorname{Im} A^{\star_L}=(\operatorname{Ker} A)^{\perp_L}$, $\operatorname{Im} A^{\star_R}=(\operatorname{Ker} A)^{\perp_R}$ and the relations \eqref{orLR} we get
\begin{align*}
(\operatorname{Im} A^{\star_L})^{\perp_R}=((\operatorname{Ker} A)^{\perp_L})^{\perp_R}=((\operatorname{Ker} A)^{\perp_R})^{\perp_L}=(\operatorname{Im} A^{\star_R})^{\perp_L}.
\end{align*}
\item[(iii)] Replacing $A$ by $A^{\star_L}$ in the relation $\operatorname{Ker} A^{\star_R}=(\operatorname{Im} A)^{\perp_R}$ and taking into account that $(A^{\star_L})^{\star_R}=A$, we obtain
\begin{align}\label{eg1}
\operatorname{Ker} (A^{\star_L})^{\star_R}=(\operatorname{Im} A^{\star_L})^{\perp_R}\Leftrightarrow \operatorname{Ker} A=(\operatorname{Im} A^{\star_L})^{\perp_R}.
\end{align}
Now substituting $A$ by $A^{\star_R}$ in the relation $\operatorname{Ker} A^{\star_L}=(\operatorname{Im} A)^{\perp_L}$ and considering that $(A^{\star_R})^{\star_L}=A$, we get
\begin{align}\label{eg2}
\operatorname{Ker} (A^{\star_R})^{\star_L}=(\operatorname{Im} A^{\star_R})^{\perp_L}\Leftrightarrow \operatorname{Ker} A=(\operatorname{Im} A^{\star_R})^{\perp_L}.
\end{align}
From \eqref{eg1} and \eqref{eg2} it follows that $\operatorname{Ker} A=(\operatorname{Im} A^{\star_L})^{\perp_R}=(\operatorname{Im} A^{\star_R})^{\perp_L}$.

Similarly we obtain $\operatorname{Im} A=(\operatorname{Ker} A^{\star_L})^{\perp_R}=(\operatorname{Ker} A^{\star_R})^{\perp_L}$.
\end{itemize}
\end{proof}

\section{Natural groups associated to geometric structures}

Given a geometric structure $b$ on $\mathbb{R}^n$ let us define two groups naturally associated to the left/right - adjoint operators,
\begin{equation*}
G_{b}^{L}=\{ A\in \operatorname{End}(\mathbb{R}_{b}^{n}):~ A^{\star_L} A =\operatorname{Id}\},~~
G_{b}^{R}=\{ A\in \operatorname{End}(\mathbb{R}_{b}^{n}):~ A^{\star_R} A=\operatorname{Id}\}.
\end{equation*}
Note that from the Definition \ref{starLR} of left/right - adjoint operators it yields
$$
G_{b}^{L}=G_{b}^{R}=G_{b}:=\{A\in \operatorname{End}(\mathbb{R}_{b}^{n}): b(A\mathbf{x},A\mathbf{y})=b(\mathbf{x},\mathbf{y}),\forall \mathbf{x},\mathbf{y}\in\mathbb{R}^{n}_{b}\}.
$$
In terms of the geometric pair $(b,B)$, using the relation $A^{\star_{L}}=B^{\top}A^{\top}(B^{\top})^{-1}$ from Proposition \ref{pimp}, we get that $\det A^{\star_L}=\det A$ and hence each $A\in G_{b}^{L}$ is invertible (since $(\det A)^2=1$) with $A^{-1}=A^{\star_L}$. Consequently, the following equivalences hold
\begin{align}\label{gr1}
 A^{\star_L} A =\operatorname{Id} \Leftrightarrow A A^{\star_L} =\operatorname{Id}  \Leftrightarrow  A B^{\top}A^{\top}(B^{\top})^{-1}=\operatorname{Id} \Leftrightarrow A B^{\top}A^{\top}=B^{\top}\Leftrightarrow A B A^{\top}=B.
\end{align}
Similarly, using the relation $A^{\star_{R}}=B A^{\top}B^{-1}$ from Proposition \ref{pimp} we obtain that each $A\in G_{b}^{R}$ is invertible (since $(\det A)^2=1$) with $A^{-1}=A^{\star_R}$ and verifies the following equivalences
\begin{align}\label{gr2}
 A^{\star_R} A=\operatorname{Id} \Leftrightarrow A A^{\star_R} =\operatorname{Id} \Leftrightarrow  A B A^{\top}B^{-1}=\operatorname{Id} \Leftrightarrow A B A^{\top}=B.
\end{align}
From the equalities \eqref{gr1} and \eqref{gr2} it follows that $$G_{b}^{L}=G_{b}^{R}=G_b=\{A\in\operatorname{End}(\mathbb{R}_{b}^{n}):~ A B A^{\top}=B\}\leq \operatorname{Aut}(\mathbb{R}_{b}^{n}).$$
Thus, $G_{b}$ is a Lie group with Lie algebra given by
$$
\operatorname{Lie}(G_b):=\mathfrak{g}_{b}=\{A\in\operatorname{End}(\mathbb{R}_{b}^{n}):~ AB=-BA^{\top}\},
$$
or equivalently
\begin{align*}
\mathfrak{g}_{b}&=\{A\in \operatorname{End}(\mathbb{R}_{b}^{n}): b(A\mathbf{x},\mathbf{y})=-b(\mathbf{x},A\mathbf{y}),\forall \mathbf{x},\mathbf{y}\in\mathbb{R}^{n}_{b}\}\\
&=\{A\in \operatorname{End}(\mathbb{R}_{b}^{n}): A^{\star_L}=-A\}=\{A\in \operatorname{End}(\mathbb{R}_{b}^{n}): A^{\star_R}=-A\}.
\end{align*}
\begin{remark} 
\begin{itemize}
\item[(i)] If $b$ is the canonical inner product on $\mathbb{R}^n$, then $G_b=O_n(\mathbb{R})$ is the orthogonal group.
\item[(ii)] If $b$ is the Minkowski inner product on $\mathbb{R}^{n}$, then $G_b=O_{n-1,1}(\mathbb{R})$ is the (full) Lorentz group.
\item[(iii)] If $b$ is the canonical symplectic structure on $\mathbb{R}^{n}$, $n\in 2\mathbb{N}$, then $G_b=Sp_{n}(\mathbb{R})$ is the symplectic group.
\end{itemize}
\end{remark}

\section{Invariant and equivariant maps}

The main purpose of this section is to analyze the compatibility between the $H-$invariant smooth real functions and the $H$-equivariant smooth vector fields on $\mathbb{R}^{n}_{b}$, where $H$ is a subgroup of $G_b$, and $b$ is an arbitrary fixed geometric structure on $\mathbb{R}^n$. In order to do that we shall need the definition of two operators naturally associated to the geometric structure $b$, which generalize the notion of gradient vector field from the Euclidean geometry, Hamiltonian vector field from the symplectic geometry, and Minkowski gradient vector field from the Minkowski geometry. More precisely, we recall the following definition from \cite{TDR}.

\begin{definition} \cite{TDR}
Let $b$ be a geometric structure on $\mathbb{R}^n$ and $\Omega\subseteq\mathbb{R}^{n}_{b}$ be an open set. Then for each $f\in\mathcal{C}^{1}(\Omega,\mathbb{R})$, the left--gradient of $f$, denoted by $\nabla^{L}_{b}f$, is the vector field uniquely defined by the relation
$$
b(\nabla^{L}_{b}f(\mathbf{x}),\mathbf{v})=\mathrm{d}f(\mathbf{x})\cdot \mathbf{v}, ~ \forall \mathbf{x}\in \Omega,\forall\mathbf{v}\in T_{\mathbf{x}}\Omega\cong \mathbb{R}^{n}_{b}.
$$
Similarly, the right--gradient of $f$, denoted by $\nabla^{R}_{b}f$, is the vector field uniquely defined by the relation
$$
b(\mathbf{v},\nabla^{R}_{b}f(\mathbf{x}))=\mathrm{d}f(\mathbf{x})\cdot \mathbf{v}, ~ \forall \mathbf{x}\in \Omega,\forall\mathbf{v}\in T_{\mathbf{x}}\Omega\cong \mathbb{R}^{n}_{b}.
$$
\end{definition}
The relation between $\nabla^{L}_{b}$, $\nabla^{R}_{b}$ in the case of a general geometric structure $b$ is provided by the following result from \cite{TDR}.
\begin{proposition}\cite{TDR}\label{rlgrd}
Let $b$ be a geometric structure on $\mathbb{R}^{n}$, and let $(b,B)$ be the associated geometric pair. Then for any fixed open set $\Omega\subseteq\mathbb{R}^{n}_{b}$, and for every $f\in\mathcal{C}^{1}(\Omega,\mathbb{R})$, we have that $\nabla^{L}_{b}f=B^{\top}\nabla f$, $\nabla^{R}_{b}f=B \nabla f$, and $\nabla^{L}_{b}f=B^{\top}B^{-1}\nabla^{R}_{b}f$, where $\nabla$ stands for the classical gradient operator.
\end{proposition}
Next we give a natural compatibility between the group $G_b$ and the operators $\nabla^{L}_{b}$ and $\nabla^{R}_{b}$. In order to do that let us recall first the definition of invariant maps and equivariant vector fields. More precisely, let $H$ be an arbitrary fixed subgroup of $G_b$. Then a real valued function $f:\mathbb{R}^{n}_b\rightarrow \mathbb{R}$ is called $H-$invariant if 
$$
f(A\mathbf{x})=f(\mathbf{x}), ~\forall\mathbf{x}\in\mathbb{R}^{n}_b, ~\forall A\in H.
$$
Similarly, a vector field $F:\mathbb{R}^{n}_b\rightarrow\mathbb{R}^{n}_b$ is called $H-$equivariant if 
$$
F(A\mathbf{x})=AF(\mathbf{x}), ~\forall\mathbf{x}\in\mathbb{R}^{n}_b, ~\forall A\in H.
$$
The following result presents the equivariance of the left/right--gradient of an invariant continuously differentiable scalar function.
\begin{theorem}\label{timp1}
Let $b$ be a geometric structure on $\mathbb{R}^n$ and let $H$ be a subgroup of $G_b$. Then for any $H-$invariant function $f\in\mathcal{C}^{1}(\mathbb{R}^{n}_b,\mathbb{R})$, the vector fields $\nabla^{L}_{b}f$, $\nabla^{R}_{b}f$ are $H-$equivariant.
\end{theorem}
\begin{proof}
From the $H-$invariance of $f$, i.e. $f(A\mathbf{x})=f(\mathbf{x}), ~\forall \mathbf{x}\in\mathbb{R}^{n}_{b}, \forall A\in H$, we get that
\begin{equation}\label{grimp}
A^{\top}\nabla f(A\mathbf{x})=\nabla f(\mathbf{x}), ~\forall \mathbf{x}\in\mathbb{R}^{n}_{b}, \forall A\in H.
\end{equation}
In order to prove the $H-$equivariance of $\nabla^{L}_{b}f$, recall from Proposition \ref{rlgrd} that $\nabla^{L}_{b}f=B^{\top}\nabla f$. Hence the relation \eqref{grimp} becomes
\begin{equation*}
A^{\top}(B^{\top})^{-1}\nabla^{L}_{b} f(A\mathbf{x})=(B^{\top})^{-1}\nabla^{L}_{b} f(\mathbf{x}), ~\forall \mathbf{x}\in\mathbb{R}^{n}_{b}, \forall A\in H.
\end{equation*}
which is equivalent to
\begin{equation}\label{eimp}
\nabla^{L}_{b} f(A\mathbf{x})=B^{\top}(A^{\top})^{-1}(B^{\top})^{-1}\nabla^{L}_{b} f(\mathbf{x}), ~\forall \mathbf{x}\in\mathbb{R}^{n}_{b}, \forall A\in H.
\end{equation}
Since every $A\in H \leq  G_b=\{A\in\operatorname{Aut(\mathbb{R}^{n}_{b})}:~ ABA^{\top}=B\}$ verifies the following equivalences
\begin{align*}
ABA^{\top}=B \Leftrightarrow B^{-1}AB=(A^{\top})^{-1}\Leftrightarrow B^{\top}A^{\top}(B^{\top})^{-1}=A^{-1}\Leftrightarrow B^{\top}(A^{\top})^{-1}(B^{\top})^{-1}=A,
\end{align*}
the relation \eqref{eimp} becomes
\begin{equation*}
\nabla^{L}_{b} f(A\mathbf{x})=A \nabla^{L}_{b} f(\mathbf{x}), ~\forall \mathbf{x}\in\mathbb{R}^{n}_{b}, \forall A\in H,
\end{equation*}
which means that $\nabla^{L}_{b} f$ is an $H-$equivariant vector field.

In order to prove the $H-$equivariance of $\nabla^{R}_{b}f$, recall from Proposition \ref{rlgrd} that $\nabla^{R}_{b}f=B\nabla f$. Hence the relation \eqref{grimp} becomes
\begin{equation*}
A^{\top}B^{-1}\nabla^{R}_{b} f(A\mathbf{x})=B^{-1}\nabla^{R}_{b} f(\mathbf{x}), ~\forall \mathbf{x}\in\mathbb{R}^{n}_{b}, \forall A\in H.
\end{equation*}
which is equivalent to
\begin{equation}\label{eimpq}
\nabla^{R}_{b} f(A\mathbf{x})=B(A^{\top})^{-1}B^{-1}\nabla^{R}_{b} f(\mathbf{x}), ~\forall \mathbf{x}\in\mathbb{R}^{n}_{b}, \forall A\in H.
\end{equation}
Since every $A\in H \leq  G_b=\{A\in\operatorname{Aut(\mathbb{R}^{n}_{b})}:~ ABA^{\top}=B\}$ verifies the following equivalences
\begin{align*}
ABA^{\top}=B \Leftrightarrow  AB=B(A^{\top})^{-1} \Leftrightarrow B(A^{\top})^{-1}B^{-1}=A,
\end{align*}
the relation \eqref{eimpq} becomes
\begin{equation*}
\nabla^{R}_{b} f(A\mathbf{x})=A \nabla^{R}_{b} f(\mathbf{x}), ~\forall \mathbf{x}\in\mathbb{R}^{n}_{b}, \forall A\in H,
\end{equation*}
which means that $\nabla^{R}_{b} f$ is an $H-$equivariant vector field.
\end{proof}
\begin{remark}
The conclusion of Theorem \ref{timp1} holds true for any $H-$invariant function $f\in\mathcal{C}^{1}(\Omega,\mathbb{R})$, where $\Omega\subseteq \mathbb{R}^{n}_b$ is open and $H-$invariant.
\end{remark}
The following result gives a correspondence between smooth equivariant vector fields on $\mathbb{R}^{n}_{b}$ and smooth invariant scalar functions on $\mathbb{R}^{n}_{b}\times\mathbb{R}^{n}_{b}$.
\begin{theorem}
Let $b$ be a geometric structure on $\mathbb{R}^n$ and $H$ be a subgroup of $G_b$. Then there exists a correspondence between smooth $H-$equivariant vector fields on $\mathbb{R}^{n}_{b}$ and smooth $H-$invariant real valued functions defined on $\mathbb{R}^{n}_{b}\times \mathbb{R}^{n}_{b}$, with respect to the diagonal action of $H$ on $\mathbb{R}^{n}_{b}\times \mathbb{R}^{n}_{b}$.
\end{theorem}
\begin{proof}
Let $F:\mathbb{R}^{n}_{b}\rightarrow \mathbb{R}^{n}_{b}$ be an $H-$equivariant smooth function. We prove that the function $f:\mathbb{R}^{n}_{b}\times \mathbb{R}^{n}_{b}\rightarrow \mathbb{R}$ given by 
\begin{equation*}
f(\mathbf{x},\mathbf{y})=b(F(\mathbf{x}),\mathbf{y}), ~\forall \mathbf{x},\mathbf{y}\in\mathbb{R}^{n}_{b},
\end{equation*}
is a smooth $H-$invariant function with respect to the diagonal action of $H$ on $\mathbb{R}^{n}_{b}\times \mathbb{R}^{n}_{b}$. Indeed, for each $A\in H\leq G_b=\{A\in \operatorname{End}(\mathbb{R}_{b}^{n}): b(A\mathbf{x},A\mathbf{y})=b(\mathbf{x},\mathbf{y}),\forall \mathbf{x},\mathbf{y}\in\mathbb{R}^{n}_{b}\}$ we have that
\begin{align*}
f(A\mathbf{x},A\mathbf{y})=b(F(A\mathbf{x}),A\mathbf{y})=b(AF(\mathbf{x}),A\mathbf{y})=b(F(\mathbf{x}),\mathbf{y})=f(\mathbf{x},\mathbf{y}), ~\forall \mathbf{x},\mathbf{y}\in\mathbb{R}^{n}_{b}.
\end{align*}
For the converse part, let $f:\mathbb{R}^{n}_{b}\times \mathbb{R}^{n}_{b}\rightarrow \mathbb{R}$ be an $H-$invariant smooth function with respect to the diagonal action of $H$ on $\mathbb{R}^{n}_{b}\times \mathbb{R}^{n}_{b}$. From Theorem \ref{timp1} it follows that the vector fields $F^{L}(\mathbf{x}):=\nabla^{L}_{b;\mathbf{y}}f(\mathbf{x},\mathbf{0}), F^{R}(\mathbf{y}):=\nabla^{R}_{b;\mathbf{x}}f(\mathbf{0},\mathbf{y}),~\forall\mathbf{x},\mathbf{y}\in\mathbb{R}^{n}_{b}$, are both $H-$equivariant.
\end{proof}

\begin{remark}\label{rem1}
Note that regardless invariance properties, for any given smooth vector field $F$ on $\mathbb{R}^{n}_{b}$, the function defined by the relation 
\begin{equation}\label{rq1}
f(\mathbf{x},\mathbf{y})=b(F(\mathbf{x}),\mathbf{y}), ~\forall \mathbf{x},\mathbf{y}\in\mathbb{R}^{n}_{b},
\end{equation}
is smooth and the defining relation \eqref{rq1} implies that $F(\mathbf{x})=\nabla^{L}_{b;\mathbf{y}}f(\mathbf{x},\mathbf{0}), ~ \forall \mathbf{x}\in\mathbb{R}^{n}_{b}$. 
Indeed, from Proposition \ref{rlgrd}, Proposition \ref{pimp} and relation \eqref{rq1}, it follows that for any $\mathbf{x},\mathbf{y}\in\mathbb{R}^{n}_{b}$ 
\begin{align*}
b(\nabla^{L}_{b;\mathbf{y}}f(\mathbf{x},\mathbf{0}),\mathbf{y})&=b(B^{\top}\nabla_{\mathbf{y}}f(\mathbf{x},\mathbf{0}),\mathbf{y})=b(B^{\star_{L}}\nabla_{\mathbf{y}}f(\mathbf{x},\mathbf{0}),\mathbf{y})=b(\nabla_{\mathbf{y}}f(\mathbf{x},\mathbf{0}),B\mathbf{y})\\
&=\langle \nabla_{\mathbf{y}}f(\mathbf{x},\mathbf{0}),\mathbf{y}\rangle=f(\mathbf{x},\mathbf{y}).
\end{align*}
Similarly, for any given smooth vector field $F$ on $\mathbb{R}^{n}_{b}$, the function defined by the relation
\begin{equation}\label{rq2}
g(\mathbf{x},\mathbf{y})=b(\mathbf{x},F(\mathbf{y})), ~\forall \mathbf{x},\mathbf{y}\in\mathbb{R}^{n}_{b},
\end{equation}
is smooth and the defining relation \eqref{rq2} implies that $F(\mathbf{y})=\nabla^{R}_{b;\mathbf{x}}g(\mathbf{0},\mathbf{y}), ~ \forall \mathbf{y}\in\mathbb{R}^{n}_{b}$.
\end{remark}
Let us now state the main result of this section which provides explicitly two sets of generators for the module of smooth equivaraint vector fields on $\mathbb{R}^{n}_{b}$ in terms of left/right--gradient vector fields.   
\begin{theorem}
Let $b$ be a geometric structure on $\mathbb{R}^n$ and $H$ a subgroup of $G_b$. Assume that the ring $\mathcal{C}^{\infty}(\mathbb{R}^{n}_{b}\times \mathbb{R}^{n}_{b},\mathbb{R})^{H}$ of $H-$invariant smooth real valued functions defined on $\mathbb{R}^{n}_{b}\times \mathbb{R}^{n}_{b}$ (with respect to the diagonal action of $H$ on $\mathbb{R}^{n}_{b}\times \mathbb{R}^{n}_{b}$) is generated by some finite set $\{u_1,\dots,u_p\}$, in the sense that for any given $f\in\mathcal{C}^{\infty}(\mathbb{R}^{n}_{b}\times \mathbb{R}^{n}_{b},\mathbb{R})^{H}$ there exists a smooth function $\tilde{f}\in\mathcal{C}^{\infty}(\mathbb{R}^{p},\mathbb{R})$ such that $f=\tilde{f}\circ(u_1,\dots,u_p)$. Then $\{\nabla^{L}_{b;\mathbf{y}}u_1(\mathbf{x},\mathbf{0}),\dots,\nabla^{L}_{b;\mathbf{y}}u_p(\mathbf{x},\mathbf{0})\}$, $\{\nabla^{R}_{b;\mathbf{x}}u_1(\mathbf{0},\mathbf{y}),\dots,\nabla^{R}_{b;\mathbf{x}}u_p(\mathbf{0},\mathbf{y})\}$, are two sets of generators of the module of smooth $H-$equivariant vector fields on $\mathbb{R}^{n}_{b}$, over the ring $\mathcal{C}^{\infty}(\mathbb{R}^{n}_{b}\times \mathbb{R}^{n}_{b},\mathbb{R})^{H}$.
\end{theorem}
\begin{proof}
Let $F:\mathbb{R}^{n}_{b}\rightarrow\mathbb{R}^{n}_{b}$ be a smooth $H-$equivariant vector field. Then as the smooth function
\begin{equation}\label{ecuu1}
f(\mathbf{x},\mathbf{y})=b(F(\mathbf{x}),\mathbf{y}), ~\forall \mathbf{x},\mathbf{y}\in\mathbb{R}^{n}_{b},
\end{equation}
is $H-$invariant with respect to the diagonal action of $H$ on $\mathbb{R}^{n}_{b}\times \mathbb{R}^{n}_{b}$, there exists a smooth function $\tilde{f}:\mathbb{R}^{p}\rightarrow \mathbb{R}$ such that 
\begin{equation}\label{equu2}
f(\mathbf{x},\mathbf{y})=\tilde{f}(u_1(\mathbf{x},\mathbf{y}),\dots,u_p(\mathbf{x},\mathbf{y})), ~\forall \mathbf{x},\mathbf{y}\in\mathbb{R}^{n}_{b}.
\end{equation}
From the relation \eqref{ecuu1} and Remark \ref{rem1} it follows that 
\begin{equation}\label{equu3}
F(\mathbf{x})=\nabla^{L}_{b;\mathbf{y}}f(\mathbf{x},\mathbf{0}), \forall\mathbf{x}\in\mathbb{R}^{n}_{b}.
\end{equation}
Using the equalities \eqref{equu2} and \eqref{equu3} we get
\begin{align*}
F(\mathbf{x})=\sum_{i=1}^{p}\partial_{i}\tilde{f}(u_1(\mathbf{x},\mathbf{0}),\dots,u_p(\mathbf{x},\mathbf{0}))\cdot\nabla^{L}_{b;\mathbf{y}}u_i(\mathbf{x},\mathbf{0}),~ \forall\mathbf{x}\in\mathbb{R}^{n}_{b},
\end{align*}
and the first part of the conclusion follows taking into account that $$\partial_{i}\tilde{f}(u_1(\mathbf{x},\mathbf{0}),\dots,u_p(\mathbf{x},\mathbf{0}))\in\mathcal{C}^{\infty}(\mathbb{R}^{n}_{b}\times \mathbb{R}^{n}_{b},\mathbb{R})^{H}, \forall i\in\{1,\dots,p\}.$$
In order to prove the second part of the conclusion, replacing in \eqref{ecuu1} the function $f$ by the smooth $H-$invariant function $g$ given by
\begin{equation*}
g(\mathbf{x},\mathbf{y})=b(\mathbf{x},F(\mathbf{y})), ~\forall \mathbf{x},\mathbf{y}\in\mathbb{R}^{n}_{b},
\end{equation*}
we obtain that 
\begin{equation}\label{equu4}
F(\mathbf{y})=\nabla^{R}_{b;\mathbf{x}}g(\mathbf{0},\mathbf{y}), \forall\mathbf{y}\in\mathbb{R}^{n}_{b}.
\end{equation}
Using similar arguments as in the previous case, there exists a smooth function $\tilde{g}:\mathbb{R}^{p}\rightarrow \mathbb{R}$ such that 
\begin{equation*}
g(\mathbf{x},\mathbf{y})=\tilde{g}(u_1(\mathbf{x},\mathbf{y}),\dots,u_p(\mathbf{x},\mathbf{y})), ~\forall \mathbf{x},\mathbf{y}\in\mathbb{R}^{n}_{b},
\end{equation*}
and hence from \eqref{equu4} we get
\begin{align*}
F(\mathbf{y})=\sum_{i=1}^{p}\partial_{i}\tilde{g}(u_1(\mathbf{0},\mathbf{y}),\dots,u_p(\mathbf{0},\mathbf{y}))\cdot\nabla^{R}_{b;\mathbf{x}}u_i(\mathbf{0},\mathbf{y}),~ \forall\mathbf{y}\in\mathbb{R}^{n}_{b}.
\end{align*}
As $$\partial_{i}\tilde{g}(u_1(\mathbf{0},\mathbf{y}),\dots,u_p(\mathbf{0},\mathbf{y}))\in\mathcal{C}^{\infty}(\mathbb{R}^{n}_{b}\times \mathbb{R}^{n}_{b},\mathbb{R})^{H}, \forall i\in\{1,\dots,p\},$$ we obtain the second part of the conclusion.
\end{proof}

\section{Laplace operators associated to geometric structures}

The aim of this section is to introduce a Laplace-like operator naturally associated to any given geometric structure on $\mathbb{R}^n$. This operator generalizes the classical Laplace operator (in the case when the geometric structure is Euclidean) and also the d'Alembert operator (in the case when the geometric structure is Minkowski). 

Starting with the classical divergence operator together with the left/right--gradient operators, let us define first two Laplace-like operators naturally associated to any arbitrary given geometric structure on $\mathbb{R}^n$.
\begin{definition}\label{lap0}
Let $b$ be a geometric structure on $\mathbb{R}^n$ and $\Omega\subseteq\mathbb{R}^{n}_{b}$ be an open set. Then for each $f\in\mathcal{C}^{2}(\Omega,\mathbb{R})$, the left--Laplacian of $f$, is given by $\Delta ^{L}_{b}f:=\operatorname{div}(\nabla ^{L}_{b}f)$, and similarly, the right--Laplacian of $f$, is given by $\Delta ^{R}_{b}f:=\operatorname{div}(\nabla ^{R}_{b}f)$.
\end{definition}
Next we prove that for any geometric structure $b$ on $\mathbb{R}^n$, the left and right Laplace operators actually coincide, and consequently they define a unique geometric Laplace operator naturally associated to the geometric structure $b$, denoted by $\Delta_b$ and called the $b-$Laplacian.
\begin{proposition}\label{deltaa}
Let $b$ be a geometric structure on $\mathbb{R}^n$ and $\Omega\subseteq\mathbb{R}^{n}_{b}$ be an open set. Then $\Delta ^{L}_{b}f=\Delta ^{R}_{b}f=:\Delta_{b}f, ~\forall f\in\mathcal{C}^{2}(\Omega,\mathbb{R})$.
\end{proposition}
\begin{proof}
Recall from Proposition \ref{rlgrd} that $\nabla^{L}_{b}f=B^{\top}\nabla f$ and $\nabla^{R}_{b}f=B \nabla f$. Denoting by $b_{ij}$, $i,j\in\{1,\dots,n\}$ the entries of the matrix representing $B$ with respect to canonical basis, it follows that
\begin{align*}
\Delta ^{L}_{b}f&=\operatorname{div}(\nabla^{L}_{b}f)=\operatorname{div}(B^{\top}\nabla f)=\sum_{1\leq i,j\leq n}b_{ji}\dfrac{\partial^{2}f}{\partial x_i \partial x_j}=\sum_{1\leq i,j\leq n}b_{ji}\dfrac{\partial^{2}f}{\partial x_j\partial x_i}\\
&=\sum_{1\leq i,j\leq n}b_{ij}\dfrac{\partial^{2}f}{\partial x_i\partial x_j}=\operatorname{div}(B\nabla f)=\operatorname{div}(\nabla^{R}_{b}f)=\Delta ^{R}_{b}f.
\end{align*} 
\end{proof}

\begin{remark}
If $b$ is the canonical inner product on $\mathbb{R}^n$ then $\Delta_{b}=\Delta$, where $\Delta$ is the classical Laplace operator.

If $b$ is the Minkowski inner product on $\mathbb{R}^n$ then $\Delta_{b}=\square$, where $\square$ is the d'Alembert operator.

If $b$ is a skew-symmetric geometric structure on $\mathbb{R}^n$, $n\in 2\mathbb{N}$, (i.e., $b$ is a symplectic structure) and $\Omega\subseteq\mathbb{R}^{n}_{b}$ is an open set, then $\Delta_{b}f\equiv  0, ~\forall f\in\mathcal{C}^{2}(\Omega,\mathbb{R})$.
\end{remark}

The following result provides some natural properties of the $b-$Laplace operator. 
\begin{proposition}\label{pimpi}
Let $b$ be a geometric structure on $\mathbb{R}^n$ and $\Omega\subseteq\mathbb{R}^{n}_{b}$ an open set. Then the following identities hold
\begin{itemize}
\item[(i)] $\Delta_{b}(\lambda f+\mu g)=\lambda \Delta_{b} f+\mu \Delta_{b} g, ~ \forall f,g\in\mathcal{C}^{2}(\Omega,\mathbb{R}),~ \forall\lambda,\mu\in\mathbb{R},$
\item[(ii)] $\Delta_{b}(fg)=f\Delta_{b}g+g\Delta_{b}f+b(\nabla_{b}^{L}f,\nabla_{b}^{L}g)+b(\nabla_{b}^{L}g,\nabla_{b}^{L}f), ~ \forall f,g\in\mathcal{C}^{\infty}(\Omega,\mathbb{R}),$
\item[(iii)] $\Delta_{b}(fg)=f\Delta_{b}g+g\Delta_{b}f+b(\nabla_{b}^{R}f,\nabla_{b}^{R}g)+b(\nabla_{b}^{R}g,\nabla_{b}^{R}f), ~ \forall f,g\in\mathcal{C}^{\infty}(\Omega,\mathbb{R}).$
\end{itemize}
\end{proposition}
\begin{proof}
$(i)$  For every $f,g\in\mathcal{C}^{2}(\Omega,\mathbb{R})$ and $\lambda,\mu\in\mathbb{R}$ we have
\begin{align*}
\Delta_{b}(\lambda f+\mu g)&=\Delta_{b}^{L}(\lambda f+\mu g)=\operatorname{div}(\nabla_{b}^{L}(\lambda f+\mu g))=\operatorname{div}(\lambda\nabla_{b}^{L} f+\mu\nabla_{b}^{L}g)\\
&=\lambda\operatorname{div}(\nabla_{b}^{L} f)+\mu \operatorname{div}(\nabla_{b}^{L} g)=\lambda \Delta_{b}^{L} f+\mu \Delta_{b}^{L} g =\lambda \Delta_{b} f+\mu \Delta_{b} g.
\end{align*}
$(ii)$  For every $f,g\in\mathcal{C}^{\infty}(\Omega,\mathbb{R})$ we have
\begin{align*}
\Delta_{b}(fg)&=\Delta_{b}^{L}(fg)=\operatorname{div}(\nabla_{b}^{L}(fg))=\operatorname{div}(g\nabla_{g}^{L}f)+\operatorname{div}(f\nabla_{g}^{L}g)=\langle\nabla g,\nabla_{b}^{L}f\rangle+g\operatorname{div}(\nabla_{b}^{L}f)\\
&+\langle\nabla f,\nabla_{b}^{L}g\rangle+f\operatorname{div}(\nabla_{b}^{L}g) =\langle\nabla g,\nabla_{b}^{L}f\rangle+\langle\nabla f,\nabla_{b}^{L}g\rangle+f\operatorname{div}(\nabla_{b}^{L}g)+g\operatorname{div}(\nabla_{b}^{L}f)\\
&=\langle (B^{-1})^{\top}\nabla_{b}^{L} g,\nabla_{b}^{L}f\rangle+\langle (B^{-1})^{\top}\nabla_{b}^{L} f,\nabla_{b}^{L}g\rangle+f\Delta_{b}^{L}g+g\Delta_{b}^{L}f=\langle \nabla_{b}^{L} g,B^{-1}\nabla_{b}^{L}f\rangle\\
&+\langle\nabla_{b}^{L} f,B^{-1}\nabla_{b}^{L}g\rangle+f\Delta_{b}^{L}g+g\Delta_{b}^{L}f=b(\nabla_{b}^{L}g,\nabla_{b}^{L}f)+b(\nabla_{b}^{L}f,\nabla_{b}^{L}g)+f\Delta_{b}^{L}g+g\Delta_{b}^{L}f\\
&=f\Delta_{b}g+g\Delta_{b}f+b(\nabla_{b}^{L}f,\nabla_{b}^{L}g)+b(\nabla_{b}^{L}g,\nabla_{b}^{L}f).
\end{align*}
$(iii)$ The proof of this item is similar to that of the item $(ii)$, this time using the right-Laplacian instead of left-Laplacian.
\end{proof}
\begin{remark}
If $b$ is a symmetric geometric structure on $\mathbb{R}^n$ (i.e., $b$ is an Euclidean or a pseudo-Euclidean inner product) and $\Omega\subseteq\mathbb{R}^{n}_{b}$ is an open set, then the relations $(ii)$ and $(iii)$ from Proposition \ref{pimpi} are identical and become
\begin{equation*}
\Delta_{b}(fg)=f\Delta_{b}g+g\Delta_{b}f+2b(\nabla_{b}f,\nabla_{b}g), ~ \forall f,g\in\mathcal{C}^{\infty}(\Omega,\mathbb{R}),
\end{equation*}
where $\nabla_b:=\nabla_{b}^{L}=\nabla_{b}^{R}$.
\end{remark}

Next we discuss some group related properties of the $b-$Laplacian. In order to do that let us give the definitions of two natural actions of the group $G_b$.

More precisely, let $b$ be a geometric structure on $\mathbb{R}^n$ and $G_b=\{A\in \operatorname{End}(\mathbb{R}_{b}^{n}): b(A\mathbf{x},A\mathbf{y})=b(\mathbf{x},\mathbf{y}),\forall \mathbf{x},\mathbf{y}\in\mathbb{R}^{n}_{b}\}$ the associated group. This group induces two natural actions: one on scalar functions defined on $\mathbb{R}^{n}_{b}$, given by 
\begin{equation}\label{act1}
(\tau(A)\cdot f)(\mathbf{x}):=f(A^{-1}\mathbf{x}), \forall\mathbf{x}\in\mathbb{R}^{n}_{b}, ~\forall A\in G_b,
\end{equation}
and other on vector fields defined on $\mathbb{R}^{n}_{b}$, given by 
\begin{equation*}
(\tilde{\tau}(A)\cdot F)(\mathbf{x}):=A F(A^{-1}\mathbf{x}), \forall\mathbf{x}\in\mathbb{R}^{n}_{b}, ~\forall A\in G_b.
\end{equation*}
Note that for any subgroup $H\leq G_b$, a scalar function $f:\mathbb{R}^{n}_{b}\rightarrow \mathbb{R}$ is $H-$invariant if and only if $\tau(A)\cdot f=f, ~\forall A\in H$, and a vector field $F:\mathbb{R}^{n}_{b}\rightarrow \mathbb{R}^{n}_{b}$ is $H-$equivariant if and only if $\tilde{\tau}(A)\cdot F=F, ~\forall A\in H$.

Let us now present a result concerning a natural compatibility between the left/right--gradient vector fields and the above defined actions.
\begin{theorem}
Let $b$ be a geometric structure on $\mathbb{R}^n$ and $f\in\mathcal{C}^{1}(\mathbb{R}^{n}_{b},\mathbb{R})$. Then the following assertions hold true:
\begin{itemize}
\item[(i)] $\nabla^{R}_{b}(\tau(A)\cdot f)=\tilde{\tau}(A)\cdot\nabla^{R}_{b}f, ~\forall A\in G_b,$
\item[(ii)] $\nabla^{L}_{b}(\tau(A)\cdot f)=\tilde{\tau}(A)\cdot\nabla^{L}_{b}f, ~\forall A\in G_b.$
\end{itemize}
\end{theorem}
\begin{proof} A direct computation shows that for any $f\in\mathcal{C}^{1}(\mathbb{R}^{n},\mathbb{R})$ the following relation holds for the classical gradient operator (i.e. the gradient operator associated to the canonical inner product $\langle\cdot,\cdot\rangle$ on $\mathbb{R}^{n}$)
\begin{equation}\label{rimpo}
\nabla(\tau(A)\cdot f)(\mathbf{x})=(A^{-1})^{\top}\nabla f(A^{-1}\mathbf{x}), ~\forall\mathbf{x}\in\mathbb{R}^{n}, ~\forall A\in\operatorname{Aut}(\mathbb{R}^{n}).
\end{equation}
$(i)$ As $G_b$ is a subgroup of $\operatorname{Aut}(\mathbb{R}^{n}_{b})=\operatorname{Aut}(\mathbb{R}^{n})$ and $\nabla^{R}_{b}f=B\nabla f$, relation \eqref{rimpo} implies that
\begin{align}\label{rimpo2}
\nabla^{R}_{b}(\tau(A)\cdot f)(\mathbf{x})=B(A^{-1})^{\top}B^{-1}\nabla^{R}_{b} f(A^{-1}\mathbf{x}),~\forall\mathbf{x}\in\mathbb{R}^{n}_{b}, ~\forall A\in G_b.
\end{align}
Taking into account the following equivalent formulation of $G_b$ $$G_b=\{A\in\operatorname{Aut}(\mathbb{R}_{b}^{n}):~ A B A^{\top}=B\},$$ we get that $$A\in G_b \Leftrightarrow  A B A^{\top}=B \Leftrightarrow B(A^{-1})^{\top}B^{-1}=A.$$ Hence, the relation \eqref{rimpo2} becomes
\begin{align*}
\nabla^{R}_{b}(\tau(A)\cdot f)(\mathbf{x})&=A \nabla^{R}_{b}f(A^{-1}\mathbf{x}),~\forall\mathbf{x}\in\mathbb{R}^{n}_{b}, ~\forall A\in G_b,
\end{align*}
which is equivalent to
\begin{align*}
\nabla^{R}_{b}(\tau(A)\cdot f)(\mathbf{x})&=(\tilde{\tau}(A)\cdot\nabla^{R}_{b}f)(\mathbf{x}),~\forall\mathbf{x}\in\mathbb{R}^{n}_{b}, ~\forall A\in G_b.
\end{align*}
$(ii)$ The proof is similar to that of the item $(i)$, this time using the relation $\eqref{rimpo}$ and the equality $\nabla^{L}_{b}f=B^{\top}\nabla f$.
\end{proof}

Now we have all the ingredients required in order to state the main result of this section which proves the $G_b -$equivariance of the $b-$Laplace operator.

\begin{theorem}\label{timpor}
Let $b$ be a geometric structure on $\mathbb{R}^n$ and $f\in\mathcal{C}^{2}(\mathbb{R}^{n}_{b},\mathbb{R})$. Then the following relation holds true:
\begin{equation*}
\Delta_{b}(\tau(A)\cdot f)=\tau(A)\cdot\Delta_{b}f, ~\forall A\in G_b.
\end{equation*}
\end{theorem}
\begin{proof}
We start by recalling the classical relation
\begin{equation}\label{diveq}
\operatorname{div}(\tilde{\tau}(A)\cdot F)=\tau(A)\cdot\operatorname{div}F,~ \forall A\in\operatorname{Aut}(\mathbb{R}^{n}),
\end{equation}
where $F:\mathbb{R}^{n}\rightarrow \mathbb{R}^{n}$ is an arbitrary continuously differentiable vector field on $\mathbb{R}^{n}$. 

As $\Delta_{b}=\Delta^{R}_{b}=\Delta^{L}_{b}$, using the relation \eqref{diveq} we get that for each $A\in G_b \leq \operatorname{Aut}(\mathbb{R}^{n})$ the following equalities hold
\begin{align*}
\Delta_{b}(\tau(A)\cdot f)&=\Delta^{R}_{b}(\tau(A)\cdot f)=\operatorname{div}(\nabla^{R}_{b}(\tau(A)\cdot f))=\operatorname{div}(\tilde{\tau}(A)\cdot\nabla^{R}_{b}f)\\
&=\tau(A)\cdot\operatorname{div}(\nabla^{R}_{b}f)=\tau(A)\cdot\Delta^{R}_{b}f=\tau(A)\cdot\Delta_{b}f.
\end{align*}
\end{proof}

\begin{remark} Assuming that $u=u(\mathbf{x})$ is a $b-$harmonic function (i.e. $\Delta_b u=0$) on a $G_b-$invariant open set, it follows from Theorem \ref{timpor} that $v_{A}=u(A\mathbf{x})$ is also $b-$harmonic for every $A\in G_b$. In particular we recover the following classical results.
\begin{itemize}
\item[(i)] If $b$ is the canonical inner product on $\mathbb{R}^n$ it follows that $\Delta_b=\Delta$ is the classical Laplace operator,  $G_b=O_{n}(\mathbb{R})$ is the orthogonal group and hence we recover the orthogonal invariance of the Laplace equation, $\Delta u=0$.
\item[(ii)] If $b$ is the Minkowski inner product on $\mathbb{R}^{n}$, it follows that $\Delta_b=\square$ is the d'Alembert operator,  $G_b=O_{n-1,1}(\mathbb{R})$ is the (full) Lorentz group and hence we recover the Lorentz invariance of the $(n-1)-$dimensional wave equation, $\square u=0$.
\end{itemize}
\end{remark}


\bigskip
\bigskip

\noindent {\sc R.M. Tudoran}\\
West University of Timi\c soara\\
Faculty of Mathematics and Computer Science\\
Department of Mathematics\\
Blvd. Vasile P\^arvan, No. 4\\
300223 - Timi\c soara, Rom\^ania.\\
E-mail: {\sf razvan.tudoran@e-uvt.ro}\\
\medskip

\end{document}